\documentclass[a4paper,12pt,oneside]{amsart}       

\usepackage[british,english]{babel} 
\usepackage{amsmath,amssymb,amsthm,amscd}   
\usepackage{mathrsfs}
\usepackage[all]{xy}
\usepackage{stmaryrd}

\newtheorem{intro_thm}{Theorem}

\theoremstyle{cited}
\newtheorem{teor}{Theorem}[section]
\newtheorem{lem}[teor]{Lemma}

\newtheorem{prop}[teor]{Proposition}

\theoremstyle{definition}
\newtheorem{deft}[teor]{Definition}

\theoremstyle{remark}

\newcommand{\psl}{PSL(2,\mathbb{C})}
\newcommand{\psln}{PSL(n,\mathbb{C})}
\newcommand{\proj}{\mathbb{P}^1(\mathbb{C})}
\newcommand{\flg}{\mathscr{F}(n,\mathbb{C})}
\newcommand{\hyp}{\mathbb{H}}
\newcommand{\Vol}{\textup{Vol}}

\title[Rigidity at infinity for the tetrahedral lattice]{Rigidity at infinity for the Borel function of the tetrahedral reflection lattice}

\author{Alessio Savini}

\begin{document}

\maketitle

\begin{abstract}
If $\Gamma$ is the fundamental group of a complete finite volume hyperbolic $3$-manifold, Guilloux \cite{guilloux:articolo} conjectured that the Borel function on the $PSL(n,\mathbb{C})$-character variety of $\Gamma$ should be rigid at infinity, that is it should stay bounded away from its maximum at ideal points. 

In this paper we prove Guilloux's conjecture in the particular case of the reflection group associated to a regular ideal tetrahedron of $\mathbb{H}^3$. 
\end{abstract}


\section{Introduction}

Let $\Gamma$ be the fundamental group of a finite volume complete hyperbolic $3$-manifold $M$. In the attempt to explore the rigidity properties of $\Gamma$, many mathematicians studied the space of representations of $\Gamma$ inside a semisimple Lie group $G$. For instance, when $G=\psln$, Bucher, Burger and Iozzi \cite{iozzi:articolo} introduced the \emph{Borel function} on the character variety $X(\Gamma,\psln)$ using bounded cohomology techniques. The Borel function is continuous with respect to the topology of pointwise convergence and its absolute value is bounded by the volume of $M$ multiplied by a suitable constant depending on $n$. Additionally, the maximum is attained only by the conjugacy class of the representation $\pi_n \circ i$ (or by its complex conjugated), where $i:\Gamma \rightarrow \psl$ is the standard lattice embedding and $\pi_n:\psl \rightarrow \psln$ is the irreducible representation. When $n=2$ the Borel function boils down to the volume function introduced for instance by Dunfield \cite{dunfield:articolo} or Francaviglia \cite{franc04:articolo} and its rigid behaviour can be translated in terms of Mostow Rigidity Theorem \cite{mostow:articolo}. 

Beyond their intrinsic interest, the previous results have several important consequences for the birationality properties of the character variety $X(\Gamma,\psln)$. For example, both Dunfield~\cite{dunfield:articolo} and Klaff-Tillmann~\cite{tillmann:articolo} used the properties of the volume function to prove that the component of the variety $X(\Gamma,\psl)$ containing the holonomy of $M$ is birational to its image through the peripheral holonomy map, which is obtained by restricting any representation to the fundamental groups of the cusps. A similar result has been obtained by Guilloux~\cite{guilloux:articolo} for the geometric component of the $\psln$-character variety, but the author needed to conjecture that outside of an analytic neighborhood of the class of the representation $\pi_n \circ i$ the Borel function is bounded away from its maximum value. 

In this paper we focus our attention on the reflection group associated to a regular ideal tetrahedron and we prove a weak version of~\cite[Conjecture 1]{guilloux:articolo} for every $n\geq 2$.

\begin{intro_thm}\label{convergence}
Let $\Gamma$ be the reflection group associated to the regular ideal tetrahedron $(0,1,e^\frac{\pi i}{3},\infty)$ and let $\Gamma_0 < \psl$ be a torsion-free finite index subgroup of $\Gamma$. Let $\rho_k:\Gamma_0 \rightarrow \psln$ be a sequence of representations and assume that each $\rho_k$ admits an equivariant measurable map $\varphi_k:\proj \rightarrow \flg$. Suppose that $\lim_{k \to \infty} \beta_n(\rho_k)={{n+1}\choose{3}}\Vol(\Gamma_0 \backslash \mathbb{H}^3)$. Then there must exist a sequence $(g_k)_{k \in \mathbb{N}}$ of elements in $\psln$ such that for every $\gamma \in \Gamma_0$ it holds $$\lim_{k \to \infty} g_k \rho_k(\gamma) g_k^{-1}=(\pi_n \circ i) (\gamma), $$ where $i:\Gamma_0 \rightarrow \psl$ is the standard lattice embedding and $\pi_n: \psl \rightarrow \psln$ is the irreducible representation.
\end{intro_thm}

This phenomen, called \emph{rigidity at infinity}, was proved by the author and Francaviglia \cite[Theorem 1.1]{savini:articolo} for $n=2$ and any non-uniform lattice of $\psl$ (notice that the same phenomenon holds for all rank-one representations of any rank-one lattice \cite{Sav20}). However, since in that case our proof exploited the existence of natural maps for non-elementary representations (see for instance ~\cite{besson95:articolo,besson96:articolo,besson98:articolo,franc09:articolo}), we could not use the same argument here. 

For our purposes, the existence of a boundary map $\varphi_k$ is crucial. Indeed, the possibility to express the Borel invariant $\beta_n(\rho_k)$ as the integral over a fundamental domain for $\Gamma_0 \backslash \psl$ of the pullback of the Borel cocycle along the boundary map $\varphi_k$ together with the maximality hypothesis allows us to prove the existence of a suitable sequence $(g_k)_{k \in \mathbb{N}}$ of elements in $\psln$ such that the sequence $(g_k \varphi_k(\gamma \underline{\xi}))_{k \in \mathbb{N}}$ is bounded, where $\underline{\xi}=(0,1,e^\frac{\pi i}{3},\infty)$ and $\gamma$ is any element of $\Gamma_0$. The boundedness of the previous sequence implies the boundedness of  $(g_k \rho_k g_k^{-1}(\gamma))_{k \in \mathbb{N}}$ for every $\gamma \in \Gamma_0$ and hence we can conclude. 

\subsection*{Plan of the paper} The first section is dedicated to preliminary definitions. We start with the notion of bounded cohomology for a locally compact group, then we recall the definition of the Borel cocycle and of the Borel class. We finally introduce the Borel invariant for a representation $\rho:\Gamma \rightarrow \psln$ and we recall its rigidity property.
The second section is devoted to the proof of the main theorem. 

\subsection*{Acknowledgements} I would like to thank Alessandra Iozzi for having proposed me this nice problem. I am also grateful to Marc Burger and Stefano Francaviglia for the enlightening discussions and the help they gave me. I thank Michelle Bucher and Antonin Guilloux for the interest they showed about this problem. I finally thank the referees for their suggestions given to improve the quality of the paper.



\section{Preliminary definitions}

\subsection{Bounded cohomology of semisimple Lie groups}

Given a locally compact group $G$ there exist several ways to introduce the notion of continuous bounded cohomology of $G$. The standard one relies on the complex of continuous bounded functions on tuples of $G$. Since in this paper we deal only with semisimple Lie groups and their lattices, we are going to follow a different approach. Indeed, in this case, one can introduce the continuous bounded cohomology of $G$ via the complex of essentially bounded measurable functions on the Furstenberg boundary. This definition is equivalent to the standard one thanks to the work  by Burger and Monod \cite[Corollary 1.5.3]{burger2:articolo}. More generally, one can use any strong resolution of $\mathbb{R}$ via relatively injective $G$-modules to compute the continuous bounded cohomology of $G$. For a more detailed exposition about these notions, we refer the reader to Monod's book \cite{monod:libro}.

Let $G$ be a semisimple Lie group of non-compact type and let $B(G)$ be its Furstenberg boundary. The latter can be identified with $G/P$, where $P$ is a minimal parabolic subgroup of $G$. For instance, when $G=\psl$, its Furstenberg boundary is $B(G)=\proj$. Recall that $B(G)$ admits a canonical quasi-invariant measure obtained by the Haar measurable structure on the group $G$. 

We define the space of \emph{bounded measurable functions} on the Furstenberg boundary as
$$
\mathcal{B}^\infty(B(G)^{n+1},\mathbb{R}):= \{ f:B(G)^{n+1} \rightarrow \mathbb{R} \ | \ \textup{$f$ is measurable} \} \ .
$$
By introducing the usual equivalence relation $f \sim g$, where $f$ and $g$ are equivalent if and only if they coincide up to a measure zero subset, we can define the space of \emph{essentially bounded measurable functions} as 
$$
L^\infty(B(G)^{n+1},\mathbb{R}):=\mathcal{B}^\infty(B(G)^{n+1},\mathbb{R}) / \sim \ .
$$
From now on, with an abuse of notation, we are going to write only $f$ when we refer to its equivalence class $[f]_\sim$. 

The space $L^\infty(B(G)^{n+1},\mathbb{R})$ admits a natural $G$-module structure given by 
$$
(gf)(\xi_0,\cdots,\xi_n):=f(g^{-1}\xi_0,\cdots,g^{-1}\xi_n) \ ,
$$
for every element $g \in G$, every function $f \in L^\infty(B(G)^{n+1},\mathbb{R})$ and almost every $\xi_0,\cdots,\xi_n \in B(G)$. Together with the \emph{standard homogeneous coboundary operator}

\[
\delta^n:L^\infty(B(G)^{n+1},\mathbb{R}) \rightarrow L^\infty(B(G)^{n+2},\mathbb{R}) ,
\]
\[
\delta^n f(\xi_0,\ldots,\xi_{n+1})=\sum_{i=0}^{n+1} (-1)^i f(\xi_0,\ldots,\xi_{i-1},\xi_{i+1}, \ldots \xi_{n+1}) ,
\]
we obtain a cochain complex $(L^\infty(B(G)^{\bullet+1},\mathbb{R}),\delta^\bullet)$. 

If we define the space of \emph{$G$-invariant functions} as
$$
L^\infty(B(G)^{n+1},\mathbb{R})^G:=\{ f \in L^\infty(B(G)^{n+1},\mathbb{R}) \ | \ gf=f \ , \forall g \in G \} \ ,
$$
we can restrict the coboundary operators to that collection of spaces getting a subcomplex $(L(B(G)^{\bullet+1},\mathbb{R})^G;\delta^\bullet_|)$. 

\begin{deft}\label{def:continuous:bounded}
The \textit{continuous bounded cohomology} of $G$ is the cohomology of the subcomplex $(L^\infty(B(G)^{\bullet+1},\mathbb{R})^G;\delta^\bullet_|)$ and it is denoted by $H^\bullet_{cb}(G,\mathbb{R})$. In a similar way, if $\Gamma < G$ is a lattice, its bounded cohomology groups are given by the cohomology of the subcomplex $(L^\infty(B(G)^{\bullet+1},\mathbb{R})^\Gamma;\delta^\bullet_|)$ and they are denoted by $H^\bullet_b(\Gamma,\mathbb{R})$. 
\end{deft}

Notice that in the case of a lattice we omitted the subscript $c$, since the topology inherited by $\Gamma$ from $G$ is the discrete one and the continuity issue becomes trivial. For both the group $G$ and its lattices, from now on, we are going to omit the real coefficients when we refer to the continuous bounded cohomology groups. 

Remarkably, one can consider the complex of bounded measurable functions $(\mathcal{B}^\infty(B(G)^{\bullet+1},\mathbb{R}),\delta^\bullet)$ to gain precious information about the continuous bounded cohomology of $G$. Here $\delta^\bullet$ still denotes the standard coboundary operator. 

\begin{prop}\cite[Proposition 2.1]{burger:articolo}\label{prop:strong}
 If we add to the complex $(\mathcal{B}^\infty(B(G)^{\bullet+1},\mathbb{R}),\delta^\bullet)$ the inclusion of coefficient $\mathbb{R} \hookrightarrow \mathcal{B}^\infty(B(G),\mathbb{R})$, 
we get back a strong resolution of $\mathbb{R}$. Hence there exists a canonical map 
\[
\mathfrak{c}^\bullet:H^\bullet(\mathcal{B}^\infty(B(G)^{\bullet+1},\mathbb{R})^G) \rightarrow H^\bullet_{cb}(G).
\]
\end{prop}

We conclude the section by observing that both Definition \ref{def:continuous:bounded} and Proposition \ref{prop:strong} are still valid if we consider the subcomplex of alternating cochains. Recall that an essentially bounded function or a bounded measurable function $f:B(G)^{n+1} \rightarrow \mathbb{R}$ is \emph{alternating} if for every permutation $\sigma \in S_{n+1}$ it holds
$$
f(x_{\sigma(0)},\ldots,x_{\sigma(n)})=\textup{sgn}(\sigma)f(x_0,\ldots,x_n) \ ,
$$
where $\textup{sgn}$ is the sign of the permutation.

\subsection{The Borel cocycle}

A \textit{complete flag} $F$ of $\mathbb{C}^n$ is a sequence of nested subspaces
\[
F^0 \subset F^1 \subset  \ldots F^{n-1} \subset F^n
\]
where $\dim_{\mathbb{C}} F^i=i$ for $i=1,\ldots,n$. Let $\flg$ be the space parametrizing all the possible complete flags of $\mathbb{C}^n$. This is a complex variety which can be thought of as a homogeneous space obtained as the quotient of $\psln$ by any of its Borel subgroups. In this way $\flg$ is the realization of the Furstenberg boundary associated to $\psln$. 

An \textit{affine flag} $(F,v)$ of $\mathbb{C}^n$ is a complete flag $F$ together with a decoration $v=(v^1,\ldots,v^n) \in (\mathbb{C}^n)^n$ such that
\[
F^i=\mathbb{C} v^i+F^{i-1}
\]
for $i=1,\ldots n$. For any $4$-tuple of affine flags $\mathbf{F}=((F_0,v_0),\ldots,(F_3,v_3))$ of $\mathbb{C}^n$ and given a multi-index $\mathbf{J} \in \{ 0,\ldots,n-1\}^4$, we set

\[
\mathcal{Q}(\mathbf{F},\mathbf{J}):=\Bigg[ \frac{\langle F_0^{j_0+1},\ldots,F_3^{j_3+1} \rangle}{\langle F_0^{j_0},\ldots,F_3^{j_3} \rangle};(v_0^{j_0+1},\ldots,v_3^{j_3+1})\Bigg] . 
\]

In the notation above we denoted by $[V,(x_0,\ldots,x_k)]$ the equivalence class of a complex $m$-dimensional vector space $V$ together with a $(k+1)$-tuple of spanning vectors $(x_0,\ldots,x_k) \in V^{k+1}$ modulo the diagonal action of $GL(m,\mathbb{C})$. 

Since the hyperbolic volume function $\textup{Vol}:\proj^4 \rightarrow \mathbb{R}$ can be thought of as defined on $(\mathbb{C}^2 \setminus \{0\})^4$, we can actually extend it by zero on the whole $(\mathbb{C}^2)^4$. Using such an extension, we define the cocycle $B_n$ as 
\begin{equation}\label{eq:borel:quotients}
B_n((F_0,v_0),\ldots,(F_3,v_3)):=\sum_{\mathbf{J} \in \{0,\ldots,n-1\}^4} \Vol \mathcal{Q}(\mathbf{F},\mathbf{J}).
\end{equation}
where we are considering the volume function exactly when the dimension of the vector space appearing in $\mathcal{Q}(\mathbf{F},\mathbf{J})$ is equal to $2$ and we set the value of the volume equal to zero otherwise. 

In the particular case of generic flags (see Definition \ref{def:general:position}), the definition of the Borel cocycle is given by Goncharov \cite{Gon93}. Its extension to the whole space of $4$-tuples of flags is due to Bucher, Burger and Iozzi, who proved the following 

\begin{prop}\cite[Corollary 13, Theorem 14]{iozzi:articolo}\label{prop:cocycle}
The function $B_n$ does not depend on the decoration used to compute it and hence it descends naturally to a function 
$$
B_n:\flg^4 \rightarrow \mathbb{R} \ ,
$$
on $4$-tuples of flags which is defined everywhere. Moreover that function is a measurable $\psln$-invariant alternating cocycle whose absolute value is bounded by ${n+1 \choose 3}\nu_3$, where $\nu_3$ is the volume of a positively oriented regular ideal tetrahedron in $\mathbb{H}^3$.
\end{prop}

As a consequence of Proposition \ref{prop:strong}, the function $B_n$ determines naturally a bounded cohomology class in $H^3_{cb}(\psln)$, which we are going to denote by $\beta_b( n)$. 
\begin{deft}
The cocycle $B_n$ is called \textit{Borel cocycle} and the class $\beta_b(n)$ is called \textit{bounded Borel class}. 
\end{deft}

Bucher, Burger and Iozzi \cite[Theorem 2]{iozzi:articolo} proved that the cohomology group $H^3_{cb}(\psln)$ is a one dimensional real vector space generated by the bounded Borel class. This generalizes a previous result by Bloch \cite{bloch:libro} for $\psl$.

We are going now to recall the main rigidity property of the Borel cocycle. Denote by $\mathcal{V}_n:\proj \rightarrow \flg$ the Veronese map. Recall that, if $\mathcal{V}_n^i(\xi)$ is the $i$-dimensional space of the flag $\mathcal{V}_n(\xi)$ and $\xi$ has homogeneous coordinates $[x:y]$, then we define $\mathcal{V}_n^{n-i}(\xi)$ as the $(n-i)$-dimensional subspace with basis  
\[
\left( 0, \ldots, 0, x^i, {{i}\choose{1}}x^{i-1}y,\ldots, {{i}\choose{j}}x^{i-j}y^j, \ldots, {{i}\choose{i-1}}xy^{i-1},y^i,0,\ldots,0 \right)^T
\]
where the first are $k$ zeros and the last are $n-i-k-1$ zeros, for $k=0,\ldots,n-1-i$. 

\begin{deft}
Let $(F_0,\ldots,F_3) \in \flg^4$ be a $4$-tuple of flags. We say that the $4$-tuple is \textit{maximal} if
\[
|B_n(F_0,\ldots,F_3)|={{n+1}\choose{3}}\nu_3 .
\]
\end{deft}
 
Maximal flags can be described in terms of the Veronese embedding. More precisely, it holds the following

\begin{teor}\cite[Theorem 19, Corollary 20]{iozzi:articolo}\label{teor:maximal:flags}
Let $(F_0,F_1,F_2,F_3)$ be a maximal $4$-tuple of flags in $\flg$. Then there must exist a unique element $g \in \psln$ such that
\[
g(F_0,F_1,F_2,F_3)=(\mathcal{V}_n(0),\mathcal{V}_n(1),\mathcal{V}_n(\pm e^{\frac{i\pi}{3}}),\mathcal{V}_n(\infty))
\]
where the sign $\pm$ reflects the sign of $B_n(F_0,\ldots,F_3)$. Additionally if $(F_0,F_1,F_2,F_3)$ and $(F_0,F_1,F_2,F'_3)$ are both maximal with the same sign, then $F_3=F_3'$.
\end{teor}

Now we discuss the continuity property of the Borel cocycle. The latter is measurable and not continuous since for instance one can consider a maximal $4$-tuple of flags $(F_0,F_1,F_2,F_3)$ and apply the sequence $(\pi_n(g)^k)_{k \in \mathbb{N}}$ to it, where $g \in \psl$ is loxodromic and $\pi_n:\psl \rightarrow \psln$ is the irreducible representation. In this way we get a sequence of maximal $4$-tuples which degenerates at the limit and for that sequence the Borel cocycle is not continuous. 

Nevertheless one can say something relevant about continuity when a $4$-tuple of flags $(F_0,F_1,F_2,F_3)$ satisfies a particular condition called \emph{general position}. 

\begin{deft}\label{def:general:position} 
Let $(F_0,F_1,F_2,F_3) \in \flg^4$ be a $4$-tuple of flags. We say that the flags are in \emph{general position} if 
$$
\dim_{\mathbb{C}} \langle F_0^{j_0}, \ldots F^{j_3}_3 \rangle=j_0 + \ldots + j_3 \ ,
$$
whenever $j_0 +  \ldots + j_3 \leq n$.
\end{deft}

For a $4$-tuple of flags in general position and a multi-index $\mathbf{J}$ such that $j_0+\cdots+j_3=n-2$, the projection of the $4$-tuple $(v_0^{j_0+1},\cdots,v_3^{j_3+1})$ to the $2$-dimensional vector space appearing in $\mathcal{Q}(\mathbf{F},\mathbf{J})$ gives us back a $4$-tuple of distinct points on a projective line. Since such a $4$-tuple varies continuously and the volume function $\Vol$ is continuous on $4$-tuples of distinct points in $\proj$, we get that the Borel cocycle is continuous on $\psln$-orbits of $4$-tuples of flags in general position.

The Borel cocycle can be used to understand when $4$ flags are in general position. Indeed we have the following

\begin{lem}\label{lemma:general:position}
Let $(F_0,F_1,F_2,F_3) \in \flg^4$ be a $4$-tuple of flags. If  
$$
|B_n(F_0,F_1,F_2,F_3) - {n+1 \choose 3}\nu_3| < \varepsilon
$$
for some $\varepsilon >0$ sufficiently small, then the flags are in general position.
\end{lem}

\begin{proof}
We are going to denote by $C_k(n)$ the number of all the possible partitions of $n$ by $k$ integers. 

Our proof will follow the line of \cite[Lemma 15]{iozzi:articolo}. We will argue by induction on $n$. Suppose $n=2$. The flags boil down to lines in $\mathbb{C}^2$ and those lines are in general position only if they are distinct. Since the Borel invariant is equal to zero when evaluated at two lines that coincide, the claim follows. 

Assume now that the statement is true for $n-1$. Given a flag $F \in \flg$ we are going to denote by $\overline F \in \mathscr{F}(\mathbb{C}^n/ \langle F^1_0 \rangle)$ the complete flag of the quotient $\mathbb{C}^n/\langle F^1_0 \rangle$ obtained by projecting $F$. Take the minimal value $j$ such that $F^1_0 \subset F^j_1$. 

We define the sets
\begin{align*}
\mathcal{J}_1&:=\{ \mathbf{J} \in \{0,\ldots,n-1\}^4 | j_0=j_1=0, 0 \leq j_2,j_3\leq n-2\} , \\
\mathcal{J}_2&:=\{ \mathbf{J} \in \{0,\ldots,n-1\}^4| j_0=0,0<j_1\leq n-2, 0 \leq j_2,j_3 \leq n-2\} , \\
\mathcal{J}_3&:=\{ \mathbf{J} \in \{0,\ldots,n-1\}^4|0<j_0\leq n-2, 0 \leq j_1,j_2,j_3 \leq n-2\} .
\end{align*}
By following the same computation of Bucher, Burger and Iozzi \cite[Equation 8, Lemma 17]{iozzi:articolo}, we have 
\begin{align}\label{eq:sum:sets}
\varepsilon &> {n+1 \choose 3}\nu_3 - B_n(F_0,\ldots,F_3)=\\
&=C_4(n-2)\nu_3 - \sum_{\mathbf{J} \in \{0,\ldots,n-1\}^4} \textup{Vol}\mathcal{Q}(\mathbf{F},\mathbf{J})= \nonumber\\
&= \left(C_4(n-3)\nu_3 - \sum_{\mathcal{J}_3} \textup{Vol}\mathcal{Q}(\mathbf{F},\mathbf{J})\right)+\left(C_3(n-3)\nu_3-\sum_{\mathcal{J}_2} \textup{Vol}\mathcal{Q}(\mathbf{F},\mathbf{J})\right)+ \nonumber \\
&+\left(C_2(n-2)\nu_3 - \sum_{\mathcal{J}_1} \textup{Vol}\mathcal{Q}(\mathbf{F},\mathbf{J})\right) \,  \nonumber
\end{align}
where we used the fact that $C_4(n-2)={n+1 \choose 3}$ and the recursive relation $C_k(n)=C_{k-1}(n)+C_k(n-1)$. Notice that in the last line of the equation we removed the vanishing terms whose multi-index $\mathbf{J}$ does not lie in any $\mathcal{J}_i$ for $i=1,2,3$.

It follows that if the Borel invariant is $\varepsilon$-near to its maximal value, then the sums over the sets $\mathcal{J}_1,\mathcal{J}_2,\mathcal{J}_3$ are $\varepsilon$-near to their maximal values. By the symmetry in the roles played by the indices appearing in $\mathbf{J}$, we must have
$$
\sum_{j_0=j_2=0,\  0 \leq j_1,j_3\leq n-2} \textup{Vol}\mathcal{Q}(\mathbf{F},\mathbf{J}) > C_2(n-2)\nu_3 - \varepsilon \ . 
$$

Using the particular choice of $j$ and following the same argument of \cite[Lemma 15]{iozzi:articolo}, we get that 
$$
(j-1)\nu_3 \geq C_2(n-2)\nu_3 -\varepsilon=(n-1)\nu_3 - \varepsilon,
$$
and since $\varepsilon$ is sufficiently small and $j$ is an integer, $j$ must be equal to $n$. This implies that $F^1_0 \subset F^n_1 \setminus F^{n-1}_1$. A similar condition holds also for $F_2$ and $F_3$. In this way we get that 
$$
\overline{F^i_k}=\overline{F_k}^i \ ,
$$
for $k=1,2,3$ and $0 \leq i \leq n-1$, whereas 
$$
\overline{F^i_0}=\overline{F_0}^{i-1} \ ,
$$
for $i \geq 1$. 

Consider now $0 \leq j_0,j_1,j_2,j_3 \leq n$ so that $j_0+j_1+j_2+j_3 \leq n$. The case $j_0=\ldots=j_3=0$ is trivial, so we will assume $j_0 \geq 1$. By Equation \eqref{eq:sum:sets} we know that the sum over $\mathcal{J}_3$ is $\varepsilon$-near to its maximal value $C_4(n-3)\nu_3$. Thanks to \cite[Equation 9]{iozzi:articolo}, we can write
$$
B_{n-1}(\overline{F_0},\ldots,\overline{F_3})=\sum_{\mathcal{J}_3}\textup{Vol}\mathcal{Q}(\mathbf{F},\mathbf{J}) \geq C_4(n-3)\nu_3 -\varepsilon \ ,
$$
hence $\overline{F_0},\ldots,\overline{F_3}$ are in general position by the inductive hypothesis. In this way we get
\begin{align*}
\textup{dim}_{\mathbb{C}}\langle F_0^{j_0} ,\ldots, F_3^{j_3} \rangle&=\textup{dim}_{\mathbb{C}}\langle \overline{F_0^{j_0}} ,\ldots,  \overline{F_3^{j_3}} \rangle+1=\\ 
&=\textup{dim}_{\mathbb{C}}\langle \overline{F_0}^{j_0-1} ,\overline{F_1}^{j_1},\ldots, \overline{F_3}^{j_3} \rangle+1=\\
&=(j_0-1)+j_1+j_2+j_3+1=\\
&=j_0+j_1+j_2+j_3 \ ,
\end{align*}
and this finishes the proof of the lemma. 

\end{proof}

The previous result is crucial in the proof of the following 

\begin{lem} \label{lemma:maximal:limit}
Let $(F_0^{(k)},\ldots,F_3^{(k)})_{k \in \mathbb{N}}$ be a sequence of $4$-tuples of flags such that 
$$
\lim_{k \to \infty} B_n(F_0^{(k)},\ldots,F_3^{(k)})={n+1 \choose 3}\nu_3 \ .
$$
Given a positively oriented regular ideal tetrahedron $\underline{\xi}=(\xi_0,\ldots,\xi_3)$, there exists a sequence $(g_k)_{k \in \mathbb{N}}$ of elements $g_k \in \psln$ such that 
$$
\lim_{k \to \infty} g_kF_i^{(k)}=\mathcal{V}_n(\xi_i) \ ,
$$
for $i=0,\ldots,3$. 
\end{lem}

\begin{proof}
By hypothesis we know that for $k$ large enough it holds
$$
| B_n(F^{(k)}_0,\ldots,F^{(k)}_3)- {n+1 \choose 3}\nu_3| < \varepsilon \ ,
$$
for $\varepsilon > 0$ fixed. By Lemma \ref{lemma:general:position}, up to discarding the first terms of the sequence, we can suppose that $F_0^{(k)},F_1^{(k)},F_2^{(k)},F^{(k)}_3$ are in general position. If $F_0,F_1$ are flags and $L$ is a line, using the transitivity of $\psln$ on triples $(F_0,F_1,L)$ in general position \cite[Lemma 23]{iozzi:articolo}, we can find a unique element $g_k \in \psln$ such that
$$
g_kF_0^{(k)}=\mathcal{V}_n(\xi_0), \ \ \ g_kF_1^{(k)}=\mathcal{V}_n(\xi_1), \ \ \ g_k(F_2^{(k)})^1=\mathcal{V}_n^1(\xi_2) \ .
$$
On the subset of $4$-tuples of flags $(F_0,F_1,F_2,F_3)$ in general position such that $F_0=\mathcal{V}_n(\xi_0)$, $F_1=\mathcal{V}_n(\xi_1)$ and $F^1_2=\mathcal{V}_n^1(\xi_2)$ the Borel cocycle is continuous (since we fixed a set of representatives in the $\psln$-orbits) and thus we argue that
$$
\lim_{k \to \infty} g_k(F_3^{(k)})^1=\mathcal{V}_n^1(\xi_3) \ ,
$$
Imitating the inductive argument in the proof of \cite[Theorem 19]{iozzi:articolo} one can show that the same holds for the other subspaces of the flags $F_2^{(k)}$ and $F^{(k)}_3$. 
\end{proof}

\subsection{The Borel invariant for representations into $\psln$}

Let $\Gamma$ be a non-uniform lattice of $\psl$ without torsion and let $\rho:\Gamma \rightarrow \psln$ be a representation. Define $M:=\Gamma \backslash \mathbb{H}^3$. It is well known that we can decompose the manifold $M$ as  $M = N \cup \bigcup_{i=1}^h C_i$, where $N$ is a compact core of $M$ and for every $i=1,\ldots,h$ the component $C_i$ is a cuspidal neighborhood diffeomorphic to $T_i \times (0,\infty)$, where $T_i$ is a torus. Since the fundamental group of the boundary $\partial N$ is abelian, the maps $i^*_b:H^k_b(M,M \setminus N) \rightarrow H^k_b(M)$ induced at the level of bounded cohomology groups are isometric isomorphisms for $k \geq 2$ (see~\cite{bucher:articolo}). Moreover, it holds $H^k_b(M,M \setminus N) \cong H^k_b(N, \partial N)$ by homotopy invariance of bounded cohomology. If we denote by $c:H^k_b(N,\partial N) \rightarrow H^k(N,\partial N)$ the comparison map, we can consider the composition
\[
\xymatrix{
H^3_{b}(\psln) \ar[r]^{\rho^*_b} & H^3_b(\Gamma) \cong H^3_b(M) \ar[r]^{\hspace{15pt}(i^*_b)^{-1}} & H^3_b(N,\partial N) \ar[r]^c & H^3(N,\partial N),
}
\]
where the isomorphism that appears in this composition holds by Gromov's Mapping Theorem \cite{gromov:articolo}. 

\begin{deft}
The \textit{Borel invariant} associated to a representation \mbox{$\rho: \Gamma \rightarrow \psln$} is given by 
\[
\beta_n(\rho):=\langle (c \circ (i^*_b)^{-1} \circ (\rho^*_b))\beta_b(n),[N,\partial N] \rangle,
\]
where the brackets $\langle \cdot , \cdot \rangle$ indicate the Kronecker pairing and $[N,\partial N] \in H_3(N,\partial N)$ is a fixed fundamental class. 
\end{deft}

The definition of the Borel invariant $\beta_n(\rho)$ is due to Bucher, Burger and Iozzi \cite{iozzi:articolo}. One can check that $\beta_n(\rho)$ does not depend on the choice of the compact core $N$ and it can be suitably extended also to lattices with torsion. We want to remark that there exist other different approaches to the Borel invariant, for instance the one given by Dimofte, Gabella and Goncharov \cite{gabella:articolo}. However, since they are all equivalent, we will consider \cite{iozzi:articolo} as our main reference.

The Borel invariant $\beta_n(\rho)$ remains unchanged on the $\psln$-conjugacy class of a representation $\rho$, hence it defines naturally a function on the character variety $X(\Gamma,\psln)$ which is continuous with respect to the topology of the pointwise convergence (this is a consequence of Proposition \ref{prop:formula}, for instance). This function, called Borel function, satisfies a strong rigidity property.

\begin{teor}\cite[Theorem 1]{iozzi:articolo}. 
Given any representation $\rho:\Gamma \rightarrow \psln$ we have 

\[
|\beta_n(\rho)| \leq {{n+1}\choose{3}}\Vol(M)
\]
and the equality holds if and only if $\rho$ is conjugated to $\pi_n \circ i$ or to its complex conjugate $\overline{\pi_n \circ i}$, where $i:\Gamma \rightarrow \psl$ is the standard lattice embedding and $\pi_n:\psl \rightarrow \psln$ is the irreducible representation. 
\end{teor}

We want to conclude this section by expressing the Borel invariant in terms of boundary maps between Furstenberg boundaries. 
We first recall the definition of the transfer map $\textup{trans}_\Gamma:H^3_b(\Gamma) \rightarrow H^3_{cb}(\psl)$. We can define the map 
\[
\textup{trans}_\Gamma:L^\infty(\proj^{n+1},\mathbb{R})^\Gamma \rightarrow L^\infty(\proj^{n+1},\mathbb{R})^{\psl}
\]
\[
\textup{trans}_\Gamma(c)(x_0,\ldots,x_n):=\int_{\Gamma \backslash PSL(2,\mathbb{C})} c(\bar gx_0,\ldots,\bar gx_n)d\mu(\bar g),
\] 
where $\bar g$ stands for the equivalence class of $g$ in $\Gamma \backslash PSL(2,\mathbb{C})$ and $\mu$ is any invariant probability measure on $\Gamma \backslash PSL(2,\mathbb{C})$. Since $\textup{trans}_\Gamma$ is a cochain map, we get a well-defined map
\[
\textup{trans}_\Gamma: H^\bullet_b(\Gamma) \rightarrow H^\bullet_{cb}(PSL(2,\mathbb{C})).
\]

Given a representation $\rho:\Gamma \rightarrow \psln$ we can consider the composition 
\[
\xymatrix{
H^3_{cb}(\psln) \ar[r]^{\hspace{20pt}\rho^*_b} & H^3_b(\Gamma)\ar[r]^{\hspace{-25pt}\textup{trans}_\Gamma} & H^3_{cb}(\psl).
}
\]
We have the following 
\begin{prop}\cite[Proposition 26, Proposition 28]{iozzi:articolo}\label{prop:formula}
Considering the composition of the map $\rho^\ast_b$ with the transfer map $trans_\Gamma$, it holds
$$
(\textup{trans}_\Gamma \circ \rho^\ast_b)(\beta_b(n))=\frac{\beta_n(\rho)}{\Vol(M)}\beta_b(2) \ .
$$
Given a measurable $\rho$-equivariant map $\varphi:\proj \rightarrow \flg$, we can rewrite the above equation in terms of cochains as follows 
\begin{equation}\label{formula}
\int_{\Gamma \backslash \psl} B_n(\varphi(g\xi_0),\ldots \varphi(g\xi_3))d\mu(g)=\frac{\beta_n(\rho)}{\Vol(M)}\Vol(\xi_0,\ldots,\xi_3) \ ,
\end{equation}
for every $(\xi_0,\ldots,\xi_3) \in \proj^4$. 
\end{prop}
 

\section{Proof of the main theorem} 

In this section we are going to prove our main theorem. The proof will follow the strategy adopted by Bucher, Burger and Iozzi for proving \cite[Theorem 29]{iozzi:articolo}. 

Let $\Gamma$ be the reflection group associated to the regular ideal tetrahedron of vertices $(0,1,e^\frac{\pi i}{3},\infty) \in \proj^4$ and let $\Gamma_0 < \psl$ be a torsion-free subgroup of $\Gamma$ of finite index. From now until the end of the paper, with an abuse of notation, we are going to denote by $g$ both a general element in $\psl$ and its equivalence class in $\Gamma_0 \backslash \psl$. 

\begin{lem}\label{almostev}
Let $\Lambda$ be a torsion-free lattice of $\psl$. Suppose $\rho_k:\Lambda \rightarrow \psln$ is a sequence of representations which satisfy $\lim_{k \to \infty} \beta_n(\rho_k)={{n+1}\choose{3}}\Vol(\Lambda \backslash \mathbb{
H}^3)$. Assume there exists a measurable map $\varphi_k:\proj \rightarrow \flg$ which is $\rho_k$-equivariant. Then, up to passing to a subsequence, for almost every $g \in \textup{Isom}(\hyp^3)$ we have

\[
\lim_{k \to \infty} B_n(\varphi_k(g \xi_0),\ldots, \varphi_k(g\xi_3))={{n+1}\choose{3}}\Vol(g\xi_0,\ldots,g\xi_3),
\]
where $(\xi_0,\ldots,\xi_3) \in \proj^4$ are the vertices of a regular ideal tetrahedron.
\end{lem}

\begin{proof}

Let $(\xi_0,\ldots,\xi_3) \in \proj^4$ be the vertices of a regular ideal tetrahedron. Without loss of generality we can assume that $\Vol(\xi_0,\ldots,\xi_3)=\nu_3$. By Proposition \ref{prop:formula} we know that Equation~(\ref{formula}) holds everywhere and hence we can write
\[
\int_{\Lambda \backslash \psl} B_n(\varphi_k(g\xi_0),\ldots,\varphi_k(g\xi_3)) d\mu_{\Lambda \backslash G}(g)=\frac{\beta_n(\rho_k)}{\Vol(\Lambda \backslash \mathbb{H}^3)}\nu_3
\]
for every $k \in \mathbb{N}$, where $\mu_{\Lambda \backslash G}$ is the measure induced by the Haar measure and renormalized to be a probability measure. Since by hypothesis $\lim_{k \to \infty} \beta_n(\rho_k)={{n+1}\choose{3}}\Vol(\Lambda \backslash \mathbb{H}^3)$, by taking the limit on both sides of the equation above we get

\begin{equation}\label{intlimit}
\lim_{k \to \infty} \int_{\Lambda \backslash \psl} B_n(\varphi_k(g\xi_0),\ldots,\varphi_k(g\xi_3)) d\mu_{\Lambda \backslash G}(g)={{n+1}\choose{3}}\nu_3.
\end{equation}

Since by Proposition \ref{prop:cocycle} the Borel cocycle satisfies $|B(F_0,\ldots, F_3)| \leq {{n+1}\choose{3}}\nu_3$, we have that
\[
{{n+1}\choose{3}}\nu_3-B_n(F_0,\ldots,F_3)=|{{n+1}\choose{3}}\nu_3-B_n(F_0,\ldots,F_3)|
\]
for every $(F_0,\ldots,F_3) \in \flg^4$. If we denote by 
$$
\varphi_k^4:\Lambda \backslash \psl \rightarrow \flg^4 \ , \ \ \varphi^4_k(g):=(\varphi_k(g\xi_0),\ldots,\varphi_k(g\xi_3)) \ ,
$$
Equation (\ref{intlimit}) implies
\[
\lim_{k \to \infty}||B_n \circ \varphi_k^4 - {{n+1}\choose{3}}\nu_3||_{L^1(\Lambda \backslash \psl,\mu_{\Lambda \backslash G})}=0.
\]

Since $L^1$-convergence implies the convergence almost everywhere of a suitable subsequence \cite[Section 7]{bartle:libro}, we can extract a subsequence $(\varphi_{k_\ell})_{\ell \in \mathbb{N}}$ such that 
\[
\lim_{\ell \to \infty} B_n (\varphi_{k_\ell}(g\xi_0),\ldots,\varphi_{k_\ell}(g\xi_3))={{n+1}\choose{3}}\nu_3
\] 
for $\mu_{\Lambda \backslash G}$-almost every $g \in \Lambda \backslash \psl$. By the equivariance of the maps $\varphi_{k_\ell}$, the equality above holds for $\mu_G$-almost every $g \in \psl$. 

If $\sigma$ is a reflection along any face of $(\xi_0,\ldots,\xi_3)$, the same argument can be adapted to a tetrahedron $(\sigma\xi_0,\ldots,\sigma\xi_3)$ which has negative maximal volume $\Vol(\sigma\xi_0,\ldots,\sigma\xi_3)=-\nu_3$. Hence the statement follows.
\end{proof}

We can apply the previous theorem for a sequence of representations $\rho_k:\Gamma_0 \rightarrow \psln$ with boundary maps $\varphi_k:\proj \rightarrow \flg$ such that $\lim_{k \to \infty} \beta_n(\rho_k)={n+1 \choose 3}\Vol(\Gamma_0 \backslash \mathbb{H}^3)$. With an abuse of notation we are going to denote by $(\varphi_k)_{k \in \mathbb{N}}$ the subsequence that we get from Lemma \ref{almostev}.

Our goal now is to show that, up to translating each boundary map $\varphi_k$ by an element $g_k \in \psln$, the sequence $g_k \varphi_k$ tends to the Veronese embedding on the vertices of the tiling of $\mathbb{H}^3$ by an ideal regular simplex. Denote by $\mathcal{T}_\textup{reg} \subset \proj^4$ the subset of $4$-tuples which are the vertices of regular ideal tetrahedra. For every element $\underline{\xi}=(\xi_0,\ldots,\xi_3)$ we denote by $\Gamma_{\underline{\xi}}$ the subgroup of $\textup{Isom}(\mathbb{H}^3)$ generated by the reflections along the faces of $\underline{\xi}$.

We start with the following 

\begin{lem}\label{reflection}
Let $\underline{\xi} \in \mathcal{T}^\infty$ be a regular tetrahedron. Consider a sequence of measurable maps $\varphi_k:\proj \rightarrow \flg$. Define

\begin{equation}\label{fullset}
\mathcal{T}^\infty:=\left\{ \underline{\xi} \in \mathcal{T}_\textup{reg} | \lim_{k \to \infty}B_n(\varphi_k(\underline{\xi}))={{n+1}\choose{3}}\Vol(\underline{\xi})\right\}
\end{equation}
where $\varphi_k(\underline{\xi}):=(\varphi_k(\xi_0),\ldots,\varphi_k(\xi_3))$ for every regular tetrahedron $\underline{\xi}=(\xi_0,\ldots,\xi_3) \in \mathcal{T}_\textup{reg}$. 
Suppose that for every $\gamma \in \Gamma_{\underline{\xi}}$ we have that $\gamma\underline{\xi} \in \mathcal{T}^\infty$. Then there exists a sequence $(g_k)_{k \in \mathbb{N}}$, where each $g_k$ is an element of $\psln$, such that 

\[
\lim_{k \to \infty} g_k\varphi_k(\alpha)=\mathcal{V}_n(\alpha)
\]
for every $\alpha \in \bigcup_{i=0}^3 \Gamma_{\underline{\xi}}\xi_i$.
\end{lem}

\begin{proof}
Since by hypothesis the tetrahedron $\underline{\xi}$ is an element of $\mathcal{T}^\infty$, by Lemma \ref{lemma:maximal:limit} we can find a sequence $(g_k)_{k \in \mathbb{N}}$ of elements in $\psln$ such that
\[
\lim_{k \to \infty} g_k\varphi_k(\xi_i)=\mathcal{V}_n(\xi_i),
\]
for $i=0,\ldots,3$. 

We want now to verify that the sequence $(g_k)_{k \in \mathbb{N}}$ is the one we were looking for. In order to do this we need to verify that
\[
\lim_{k \to \infty} g_k\varphi_k(\gamma \xi_i)=\mathcal{V}_n(\gamma \xi_i)
\]
for $i=0,\cdots,3$ and for every $\gamma \in \Gamma_{\underline{\xi}}$. If $\gamma$ is an arbitrary element of $\Gamma_{\underline{\xi}}$ we can write it as $\gamma=r_N \cdot r_{N-1} \ldots \cdot r_1$, where each $r_i$ is a reflection along a face of the tetrahedron $r_{i-1}\cdot \ldots \cdot r_1 \underline{\xi}$. We are going to prove the statement by induction on $N$. If $N=0$ there is nothing to prove. Assume the statement holds for $\gamma'=r_{N-1} \cdot \ldots \cdot r_1$. Denote by $\underline{\eta}=\gamma' \underline{\xi}$. We know that for the vertices of $\underline{\eta}$ we have 
\[
\lim_{k \to \infty} g_k\varphi_k(\eta_i)=\mathcal{V}_n(\eta_i),
\]
for $i=0,\ldots,3$. We want to prove that
\[
\lim_{k \to \infty} g_k\varphi_k(r_N\eta_i)=\mathcal{V}_n(r_N\eta_i),
\]
for $i=0,\ldots,3$. Assume $r_N$ is the reflection along the face of $\underline{\eta}$ whose vertices are $\eta_1,\eta_2$ and $\eta_3$. In particular we have that $r_N \eta_i=\eta_i$ for $i=1,2,3$, so for these vertices the statement holds. We are left to prove that 
\[
\lim_{k \to \infty} g_k\varphi_k(r_N\eta_0)=\mathcal{V}_n(r_N\eta_0).
\]

The sequence $(g_k\varphi_k(r_N\eta_0))_{k \in \mathbb{N}}$ is a sequence of points in $\flg$, which is compact. Hence we can extract a subsequence which converges to a point $\alpha_0 \in \flg$. By Lemma \ref{lemma:general:position} we know that the $4$-tuple $g_k\varphi_k(\underline{\eta})$ is eventually in general position. By the continuity of the Borel cocycle on the set of $4$-tuples in general position we get
\[
\lim_{k \to \infty} B_n(g_k \varphi_k(r_N\eta_0),g_k\varphi_k(\eta_1),\ldots,g_k\varphi_k(\eta_3))=B_n(\alpha_0,\mathcal{V}_n(\eta_1),\ldots,\mathcal{V}_n(\eta_3)).
\] 
At the same time, by hypothesis it follows
\[
\lim_{k \to \infty} B_n(g_k \varphi_k(r_N\underline{\eta}))={{n+1}\choose{3}}\Vol(r_N\underline{\eta})=-{{n+1}\choose{3}}\Vol(\underline{\eta}).
\]
On the other hand, it holds
\[
B_n(\mathcal{V}_n(r_N\underline{\eta}))={{n+1}\choose{3}}\Vol(r_N\underline{\eta})=-{{n+1}\choose{3}}\Vol(\underline{\eta}).
\]
and hence, by a simple comparison argument, we get
\[
B_n(\mathcal{V}_n(r_N\eta_0),\mathcal{V}_n(\eta_1),\ldots,\mathcal{V}_n(\eta_3))=B_n(\alpha_0,\mathcal{V}_n(\eta_1),\ldots,\mathcal{V}_n(\eta_3))=\pm {{n+1}\choose{3}} \nu_3. 
\]

As a consequence we must have $\alpha_0=\mathcal{V}_n(r_N\eta_0)$, but this is equivalent to say that the sequence $(g_k\varphi_k(r_N\eta_0))_{k \in \mathbb{N}}$ satisfies
\[
\lim_{k \to \infty} g_k\varphi_k(r_N\eta_0)=\mathcal{V}_n(r_N\eta_0)
\]
for any convergent subsequence of $(g_k\varphi_k(r_N\eta_0))_{k \in \mathbb{N}}$. Then the statement follows. 
\end{proof}

We are now ready to prove the main theorem. 

\begin{proof}[Proof of Theorem~\ref{convergence}]
Define the set
\[
\mathcal{T}^\infty_\Gamma:=\{\underline{\xi} \in \mathcal{T}^\infty|\forall \gamma\in \Gamma_{\underline{\xi}}: \gamma\underline{\xi} \in \mathcal{T}^\infty\}.
\]

We claim that this set is a set of full measure in $\mathcal{T}_\textup{reg}$. By Lemma~\ref{almostev}, we already know that $\mathcal{T}^\infty$ defined by Equation~(\ref{fullset}) is a set of full measure. For any $\underline{\eta} \in \mathcal{T}_\textup{reg}$ we define the evaluation map 
\[
\textup{ev}_{\underline{\eta}}:\textup{Isom}(\mathbb{H}^3) \rightarrow \mathcal{T}_\textup{reg}, \hspace{10pt} \textup{ev}_{\underline{\eta}}(g):=g \underline{\eta}. 
\]

Set $G^\infty:=\textup{ev}_{\underline{\eta}}^{-1}(\mathcal{T}^\infty)$ and $G^\infty_\Gamma:=\textup{ev}_{\underline{\eta}}^{-1}(\mathcal{T}^\infty_\Gamma)$. Let $\underline{\xi}=g \underline{\eta}$. It holds $\underline{\xi} \in \mathcal{T}^\infty_\Gamma$ if and only if for any $\gamma \in \Gamma_{\underline{\xi}}$ we have that $\gamma \underline{\xi} = \gamma g \underline{\eta} \in \mathcal{T}^\infty$. Since $\Gamma_{\underline{\xi}}=\Gamma_{g \underline{\eta}}=g\Gamma_{\underline{\eta}}g^{-1}$, any element $\gamma \in \Gamma_{\underline{\xi}}$ can be written as $\gamma=g \gamma_0 g^{-1}$, where $\gamma_0 \in \Gamma_{\underline{\eta}}$. Thus, by a simple substitution, we get that $\underline{\xi} \in \mathcal{T}^\infty_\Gamma$ if and only if for every $\gamma_0 \in \Gamma_{\underline{\eta}}$ we have that $g\gamma_0\underline{\eta} \in \mathcal{T}^\infty$. This argument implies that we can write 
\[
G^\infty_\Gamma=\bigcap_{\gamma_0 \in \Gamma_{\underline{\eta}}} G^\infty \gamma_0^{-1}.
\]

All the sets $G^\infty \gamma_0^{-1}$ are sets of full measure, since are right-translated of the set of full measure $G^\infty$ by the element $\gamma_0^{-1}$. Being a countable intersection of full measure sets, also $G^\infty_\Gamma$ has full measure. Hence also $\mathcal{T}^\infty_\Gamma$ has full measure, as claimed.  

Since all regular ideal tetrahedra are in a unique $\textup{Isom}(\mathbb{H}^3)$-orbit, up to conjugating each representation $\rho_k$, we can assume that $\xi=(0,1,e^\frac{\pi i}{3},\infty) \in \mathcal{T}^\infty_\Gamma$. With this assumption we have that $\Gamma_{\underline{\xi}}=\Gamma$, the reflection lattice we started with. By applying Lemma~\ref{reflection}, there must exists a sequence $(g_k)_{k \in \mathbb{N}}$ of elements $g_k \in \psln$ such that
\[
\lim_{k \to \infty} g_k\varphi_k(\gamma \underline{\xi})=\mathcal{V}_n(\gamma \underline{\xi})=\pi_n(\gamma)\mathcal{V}_n(\underline{\xi})
\]
for every $\gamma \in \Gamma$ and hence for every $\gamma \in \Gamma_0$, where $\pi_n:\Gamma_0 \rightarrow \psln$ is the irreducible representation and $\mathcal{V}_n: \proj \rightarrow \flg$ is the Veronese embedding. For every $k \in \mathbb{N}$ we define $\tilde \varphi_k:=g_k\varphi_k$ and $\tilde \rho_k:=g_k \rho_k g_k^{-1}$. We get that

\[
\lim_{k \to \infty} \tilde \rho_k(\gamma)\tilde \varphi_k(\underline{\xi})=\lim_{k \to \infty} \tilde \varphi_k(\gamma \underline{\xi})=\mathcal{V}_n(\gamma \underline{\xi})=\pi_n (\gamma)\mathcal{V}_n(\underline{\xi}),
\]
for every $\gamma \in \Gamma_0$. In particular notice that both sequences $(\varphi_k(\underline{\xi}))_{k \in \mathbb{N}}$ and $(\varphi_k(\gamma \underline{\xi}))_{k \in \mathbb{N}}$ are converging. The element $\gamma$ acts as $\pi_n(\gamma)$ at the limit, hence the sequence $(\tilde \rho_k(\gamma))_{k \in \mathbb{N}}$ cannot diverge and it remains bounded in $\psln$. Hence the sequence of representations $(\tilde \rho_k)_{k \in \mathbb{N}}$ has to be bounded in the character variety $X(\Gamma_0,\psln)$ and there must exists a subsequence of $(\tilde \rho_k)_{k \in \mathbb{N}}$ converging to a suitable representation $\rho_\infty$.

By the continuity of the Borel function on the character variety $X(\Gamma_0,\psln)$ with respect to the pointwise topology, it follows
\[
\beta_n(\rho_\infty)=\lim_{k \to \infty} \beta_n(\tilde \rho_k)=\lim_{k \to \infty} \beta_n(\rho_k)={{n+1}\choose{3}}\textup{Vol}(\Gamma_0 \backslash \mathbb{H}^3).
\]

By~\cite[Theorem 1]{iozzi:articolo} the representation $\rho_\infty$ must be conjugated to the representation $(\pi_n \circ i)$, where $i:\Gamma_0 \rightarrow \psl$ is the standard lattice embedding and $\pi_n:\psl \rightarrow \psln$ is the irreducible representation. Since the argument above holds for every convergent subsequence of $(\tilde \rho_k)_{k \in \mathbb{N}}$, the theorem follows.
\end{proof}

We conclude by noticing that in the proof we exploited crucially the combinatorial structure of the reflection group $\Gamma$.  For this reason it seems unlikely to adapt the proof for more general lattices. 


\vspace{20pt}
Alessio Savini\\
Section de Math\'ematiques,\\
University of Geneva,\\
Rue Du Conseil G\'eneral 7-9,\\
Geneva 1205,\\
\texttt{alessio.savini@unige.ch}\\

\end{document}